%% file: main.tex
\documentclass[12pt]{amsart}
\usepackage[a4paper,margin=1in]{geometry}
\usepackage{microtype}
\usepackage[cal=euler]{mathalfa}
\usepackage{tikz,tikz-cd}

\input{commands}
\usepackage[colorlinks,allcolors=blue]{hyperref}

\oper{Var}
\oper{crit}
\def\vir{\mathrm{vir}}

\begin{document}
\title[DT invariants of dihedral quotients]
{Donaldson-Thomas invariants for $3$-Calabi-Yau varieties of dihedral quotient type}

\author{Sergey Mozgovoy}
\address{School of Mathematics, Trinity College Dublin, College Green, Dublin 2, Ireland}

\author{Markus Reineke}
\address{Fakultät für Mathematik, Ruhr-Universität Bochum, Universitätsstraße 150, 44780 Bochum, Germany}

\begin{abstract} 
We compute motivic Donaldson-Thomas invariants for crepant resolutions of quotients of affine three-space by even dihedral groups in terms of 
an affine type $D$ root system, using double dimensional reduction and the representation theory of affine type $D$ quivers.
\end{abstract}
\maketitle

\input{introduction}

\input{recollections}

\input{recollections2}
\input{dihedral}
\input{computation_of_sigma}
\input{generating_function}

\bibliography{biblio}
\bibliographystyle{hamsplain}
\end{document}

%% file: introduction.tex
\section{Introduction}

The Donaldson-Thomas invariants of Calabi-Yau threefolds, introduced in \cite{Thomas}, admit a motivic refinement defined in \cite{BBS,KS},  from which several other numerical invariants can be derived. 

A complete description of all motivic DT invariants is achieved only for very special classes of Calabi-Yau threefolds,
in particular,
for affine three-space \cite{BBS}, for the (resolved) conifold \cite{MMNS},
for crepant resolutions related to Kleinian singularities \cite{Mozgovoy_McKay},
and for small crepant resolutions
of affine toric Calabi-Yau $3$-folds \cite{MorrisonNagao} (except the case of $\bC^3/(\bZ_2\xx\bZ_2)$, considered in this paper).
See \cite{MP} for an overview of these results.

In the present work, we give a complete description of motivic DT invariants for crepant resolutions of quotients of affine three-space by a natural action of even dihedral groups. Our main result, Theorem \ref{thm:gen fun}, shows that the motivic DT invariants are related to 
an affine type $D$ root system,  and that only three different non-zero values of the DT invariants appear.

As in the case of the works cited above, this result is achieved by factoring the motivic generating series of a particular $3$-Calabi-Yau algebra, given as the Jacobian algebra of a quiver with potential naturally associated to the dihedral group action \cite{ginzburg}. Since the relevant potential admits a cut \cite{MMNS}, dimensional reduction can be performed, leading to a twisted version of a preprojective algebra of an extended Dynkin quiver \cite{mozgovoy_translation}. It admits a further dimensional reduction, which reduces the problem to a generating series of representations of an extended Dynkin quiver, twisted by a certain interaction form. Using the known representation theory of extended Dynkin quivers, this quadratic form can be computed on all representations, and an explicit factorization of the motivic generating function can be achieved.

We start in Section \ref{recoll} by recalling the construction of quivers with potentials from finite group actions on affine three-space, the definition of motivic generating functions and motivic Donaldson-Thomas invariants, and the method of dimensional reduction for potentials with cuts. In Section \ref{mckay} we compute the quiver with potential corresponding to the dihedral group action we are interested in, relate the algebra resulting from dimensional reduction to twisted preprojective algebras, and state an explicit form of the motivic generating function involving the crucial interaction form. Its explicit computation is the topic of Section \ref{sec:sigma}. After these preparations, we formulate and prove our main result in Section \ref{sec:genfun}, and illustrate it in the example of the action of a Klein four-group.\\[2ex]
{\it Acknowledgments:} 
The authors are supported by the project "SFB-TRR 191 Symplectic structures in geometry, algebra and dynamics" of the Deutsche Forschungsgemeinschaft.

%% file: recollections.tex
\section{Recollections on 3-CY algebras and their DT invariants}\label{recoll}

In this section, we collect the main definitions and results on $3$-Calabi-Yau algebras and their Donaldson-Thomas invariants. 

\sec[3-Calabi-Yau algebras] We follow \cite[\S4.4]{ginzburg}. Let $G$ be a finite subgroup of ${\rm SL}_3(\mathbb{C})$, and let $L_1,\ldots,L_s$ be a system of representatives of the isomorphim classes of irreducible representations of $G$. Let $V=\mathbb{C}^3$ be the defining representation of $G$, which therefore carries a $G$-invariant determinant map
$\det:V^{\otimes 3}\to\bigwedge^3V\iso\mathbb{C}$. 
We define the McKay quiver,
associated with representation $V$,
as the quiver with vertices $1,\ldots,s$, and the arrows $i\rightarrow j$ parametrized by a basis of ${\rm Hom}_G(L_i,L_j\otimes V)$. For any oriented triangle $\Delta:\, i\rightarrow j\rightarrow k\rightarrow i$ in $Q$, we define $\lambda_\Delta$  as the scalar factor of the identity map in the composition
\begin{equation}
L_i\rightarrow L_j\otimes V\rightarrow L_k\otimes V^{\otimes 2}\rightarrow L_i\otimes V^{\otimes 3}
\xto\det L_i.
\end{equation}
We define a potential for $Q$ as the linear combination
\begin{equation}
W=\sum_\Delta\lambda_\Delta\Delta,
\end{equation}
the summation ranging over all oriented triangles in $Q$ up to cyclic shift. 
Let $J=\mathbb{C}Q/\langle\partial W\rangle$ be the Jacobian algebra for the quiver with potential $(Q,W)$. 
By \cite[Theorem 4.4.6]{ginzburg}, $J$ is Morita equivalent to the skew group algebra $\mathbb{C}[V]\# G$.
By \cite{bridgeland_mukai},
its category of modules is derived equivalent to the category of coherent sheaves on any crepant resolution of the quotient $V\GIT G$.

\sec[Donaldson-Thomas invariants] 

We follow \cite{MMNS}. Let $(Q,W)$ be a pair consisting of a finite quiver $Q$, with set of vertices $Q_0$ and set of arrows $Q_1$, and a potential $W\in\mathbb{C}Q/[\mathbb{C}Q,\mathbb{C}Q]$, that is, a finite linear combination of cyclic paths considered up to cyclic shift. 
We denote by 
\begin{equation}
J=\mathbb{C}Q/\langle\partial W\rangle
=\bC Q/\angs{\dd W/\dd a}{a\in Q_1}
\end{equation}
the Jacobian algebra of $(Q,W)$. We denote by $\chi_Q$ the homological Euler form of $Q$.

For a dimension vector $d\in\mathbb{N}Q_0$, let $R(Q,d)=\bop_{a:i\to j}\Hom(\bC^{d_i},\bC^{d_j})$ be the variety of representations of $Q$ of dimension vector $d$, which carries a natural action of the base change group $G_d=\prod_{i\in Q_0}\GL_{d_i}(\bC)$.
Let $R(J,d)\sbs R(Q,d)$ be the closed subvariety of $J$-representations.
It is known that $R(J,d)=\crit(f_d)$,
the set of critial points of the function $f_d:R(Q,d)\to \bC$
given by the trace of the potential $W$ on representations of $Q$.
Consider the quotient stack $\cM(J,d)=[\crit(f_d)/G_d]$,
which is the moduli stack of representations of $J$ of dimension vector $d$.

In the following, we work in the localized Grothendieck ring
\begin{equation}
R=K_0(\Var_\bC)[q^{\pm\oh},\, (1-q^i)^{-1}\rcol i\geq 1]
\end{equation}
of complex varieties, where $q=[\bA^1]$ denotes the Lefschetz motive. In fact, all calculations in this work will be performed in the subring
$R'=\mathbb{Q}[q^{\pm\oh},\, (1-q^i)^{-1}\rcol i\geq 1]\subset\mathbb{Q}(q^\oh)$. 
The moduli stack $\mathcal{M}(J,d)$ admits a virtual motive $[\mathcal{M}(J,d)]_\vir\in R$, defined using motivic nearby cycles. We will not need the precise definition, since in the case of potentials with cuts (see \S\ref{potentialwithcut}), this virtual motive can be expressed in more elementary terms.
In the situation we are interested in, it is in fact an element of $R'$.

Let $R[Q_0]=\bop_{d\in\bN Q_0}Rt^d$ be the algebra of polynomials over $R$ and let $R\pser{Q_0}$ be its completion.
We define the (total) motivic generating series
of $(Q,W)$ as
\begin{equation}\label{total1}
\cA(t)=\sum_{d\in\bN Q_0}[\cM(J,d)]_\vir
\cdot t^d\in R\pser{Q_0}.
\end{equation}

We define the plethystic exponential $\Exp$ of a series
in $R'\pser{Q_0}$ without constant term as
\begin{equation}
\Exp\rbr{\sum_{d\in\mathbb{N}Q_0}
\sum_{k\ge 0}c_{d,k}\cdot  q^{k/2}t^d}
=\prod_{d\in\bN Q_0}\prod_{k\ge0}(1-q^{k/2}t^d)^{-c_{d,k}}.
\end{equation}

Assuming that $\cA(t)\in R'\pser{Q_0}$, we define the motivic Donaldson--Thomas invariants $\Omega_d(q)\in \mathbb{Q}(q^{\oh})$ of $(Q,W)$ by the factorization
\begin{equation}
\mathcal{A}(t)={\rm Exp}\left(\frac{\sum_d\Omega_d(q)t^d}{q-1}\right).
\end{equation}

\subsection{\tpdf{$q$}q-Pochhammer symbols}
Let us define the $q$-Pochhammer symbols
\begin{equation}
(t;q)_n=\prod_{i=0}^{n-1}(1-q^it),\qquad
(t;q)_\infty=\prod_{i=0}^{\infty}(1-q^it)
=\Exp\rbr{\frac t{q-1}}
.
\end{equation} 
They satisfy
\begin{equation}
\sum_{n\ge0}\frac{(a;q)_n}{(q;q)_n}t^n
=\frac{(at;q)_\infty}{(t;q)_\infty}
=\Exp\rbr{\frac{1-a}{1-q}t}.
\end{equation}
which leads to the versions of the $q$-binomial theorem
\begin{equation}
\sum_{n\ge0}\frac{t^n}{(q;q)_n}
=\frac{1}{(t;q)_\infty}
,\qquad
\sum_{n\ge0}\frac{(-1)^nq^{\binom n2}t^n}{(q;q)_n}=(t;q)_\infty.
\end{equation}
In what follows we will use $(q)_n=(q;q)_n=\prod_{i=1}^n(1-q^i)$.
Note that 
\begin{equation}
[\GL_n]=q^{n^2}\cdot (q\inv)_n,\qquad n\ge1.
\end{equation}



%% file: recollections2.tex

\subsection{DT invariants for potentials with cuts}
\label{potentialwithcut}
Let $(Q,W)$ be a quiver with a potential and $I\sbs Q_1$ be a cut,
meaning a subset such that every non-zero term of $W$ contains exactly one element from $I$.
Let $Q'$ be the quiver $(Q_0,Q_1\ms I)$ and
\begin{equation}
J_I=\bC Q'/\ang{\dd W/\dd a\rcol a\in I}.
\end{equation}
The required motivic generating series \eqref{total1} is equal to
\begin{equation}\label{total2}
\cA(t)=\sum_{d\in\bN Q_0}(-q^\oh)^{s(d,d)}\frac{[R(J_I,d)]}{[G_d]}t^d,\qquad
s(d,e)=\hi_Q(d,e)+2\sum_{(a:i\to j)\in I}d_ie_j
\end{equation}
and is independent of a cut. 

Let us assume that there is another cut $I'\sbs Q_1$ disjoint from $I$.
We define the quiver $Q''=(Q_0,Q_1\ms(I\cup I'))$ and consider the forgetful map
\begin{equation}\label{proj}
\pi:R(J_I,d)\to R(Q'',d)
\end{equation}
having linear fibers.
Given a $Q''$-representation $M$, let $\vi(M)$ denote the dimension of the fiber $\pi\inv(M)$.
In our applications $\vi(M)$ will be quadratic, meaning that
there exist values $\vi(M,N)$ such that $\vi(M)=\vi(M,M)$ and $\vi(\bop_i M_i,\bop_j N_j)=\sum_{i,j}\vi(M_i,N_j)$.

Let $S$ be the set parameterizing isomorphism classes of all indecomposable $Q''$-represen\-ta\-tions.
Then every representation can be written in the form $M=\bop_{X\in S}X^{\oplus m_X}$
for some map $m:S\to\bN$ with finite support.
Therefore we have
\begin{multline}
\label{part function}
\cA(t)=\sum_{d\in\bN Q_0}\sum_{\udim M=d}
\frac{(-q^\oh)^{s(d,d)+2\vi(M)}}{[\Aut M]}t^{d}\\
=\sum_{m:S\to\bN}
\frac{(-q^\oh)^{-\sum_{X,Y}m_X m_Y\si(X,Y)}}
{\prod_X(q\inv)_{m_X}}
t^{\sum_X m_X\udim X}
\end{multline}
where we used the fact that (see \cite{Mozgovoy_McKay})
\begin{equation}
[\Aut M]=[\End M]\cdot\prod_{X\in S}(q\inv)_{m_X},\qquad
M=\bop_{X\in S}X^{\oplus m_X},
\end{equation}
and where
\begin{equation}\label{sigma}
\si(M,N)=2h^0(M,N)-2\vi(M,N)-s(d,e),
\qquad d=\udim M,\,e=\udim N,
\end{equation}
and $h^i(M,N)=\dim\Ext^i(M,N)$ for $M,N\in\Rep(Q'')$.
The formula \eqref{part function},
representing the motivic generating series of $(Q,W)$ in terms of representations of $Q''$, 
is called the (double) dimensional reduction formula.

\begin{remark}
In the above formulas we can substitute $\si$ by its symmetrization.
We call the map $\si:S\xx S\to\bZ$ an interaction form and we say that $M,N$  do not interact if $\si(M,N)=0$ (for the symmetrized form $\si$).
Assume that we can decompose $S=S_1\sqcup\dots\sqcup S_r$ so that elements of $S_i$ and $S_j$  do not interact for $i\ne j$.
We can define the generating series $\cA_k(t)$ by restricting the above sum to maps $m:S_k\to\bN$.
Then we obtain $\cA(t)=\prod_k \cA_k(t)$.
In what follows we will decompose $S$ into simple classes that don't interact with each other and then we will compute generating series for separate classes.
\end{remark}

\subsection{Generating series of \tpdf{$C_2$}{C2}}
The following  result describes the motivic generating series of the cyclic quiver $C_2$ (equipped with the zero potential) having two vertices $1,2$ and arrows $1\to 2$ and $2\to 1$.
It will be used later in our study of motivic generating series associated to dihedral quotients. 

\begin{theorem}
\label{C2 part function}
The motivic generating series of the cyclic quiver $C_2$ is equal to
$$F(t)=\sum_{d\in\bN^2}
\frac{(-q^\oh)^{-(d_1-d_2)^2}}{\prod_i(q\inv)_{d_i}}t^d=\Exp\rbr{\frac{qt_1t_2-q^{\oh}(t_1+t_2)}{q-1}}.$$
\end{theorem}
\begin{proof}
Let $\hi_{C_2}$ be the Euler form of $C_2$.
Then the motivic generating series of $C_2$ is
\eqref{total2}
$$
F(t)=
\sum_{d\in\bN^2}(-q^\oh)^{\hi_{C_2}(d,d)}
\frac{[R(C_2,d)]}{[G_d]}t^d
=\sum_{d\in\bN^2}\frac{(-q^\oh)^{-\hi_{C_2}(d,d)}}{\prod_i(q\inv)_{d_i}}t^d$$
which is the left hand side of the required equality.

Let us consider the forgetful map $\pi:R(C_2,d)\to R(A_2,d)$, where $A_2$ is the quiver $1\to 2$.
Given a representation $M\in R(A_2,d)$, the corresponding fiber has dimension $\vi(M)=d_1d_2$.
Let $X_1=S_1,X_2=S_2,X_3=P_1$ be the indecomposable representations of $A_2$.
Then every $A_2$-representation $M$ can be written in the form $M=\bop_i X_i^{\bop m_i}$ for some $m\in\bN^3$.
Note that $d=\udim M=(m_1+m_3,m_2+m_3)$ and
$$h^0(M,M)=\sum_{i,j}m_im_j h^0(X_i,X_j)=\sum_i m_i^2+m_2m_3+m_3m_1=:h^0(m,m).$$
As in \eqref{part function}, we obtain
\begin{multline*}
F(t)=\sum_d\sum_{\udim M=d}
\frac{(-q^\oh)^{\hi_{C_2}(d,d)+2d_1d_2}}{[\Aut M]}t^d\\
=\sum_{\ov{m\in\bN^3}{d=(m_1+m_3,m_2+m_3)}}
\frac{(-q^\oh)^{d_1^2+d_2^2-2h^0(m,m)}}
{\prod_i(q\inv)_{m_i}}t_1^{m_1+m_3}t_2^{m_2+m_3}
=\sum_{m\in\bN^3}
\frac{(-q^\oh)^{-m_1^2-m_2^2}}{\prod_i(q\inv)_{m_i}}
t_1^{m_1}t_2^{m_2}(t_1t_2)^{m_3}.
\end{multline*}
We note that (the formulas correspond to the case of a one loop quiver and a zero loop quiver respectively; they also follow from the $q$-binomial theorem)
$$\sum_{m\ge0}\frac{t^m}{(q\inv)_m}=\Exp\rbr{\frac {qt}{q-1}},\qquad
\sum_{m\ge0}\frac{(-q^\oh)^{-m^2}}{(q\inv)_m}t^m
=\Exp\rbr{\frac {-q^\oh t}{q-1}}.
$$
Therefore
$$
F(t)=
\Exp\rbr{\frac {-q^\oh t_1}{q-1}}\cdot
\Exp\rbr{\frac {-q^\oh t_2}{q-1}}\cdot
\Exp\rbr{\frac {q t_1t_2}{q-1}}
=\Exp\rbr{\frac{qt_1t_2-q^{\oh}(t_1+t_2)}{q-1}}.$$
\end{proof}

%% file: dihedral.tex
\section{McKay quivers for dihedral groups}\label{mckay}

\subsection{McKay quiver with potential for even dihedral groups}

We first collect some elementary facts on the representation theory of even dihedral groups. 
So we consider the group $D_{2\ell}=\bZ_2\ltimes\bZ_{2\ell}$ with 
$\ell\geq 1$.
It admits the standard presentation $$D_{2\ell}=\angs{d,s}{s^2=d^{2\ell}=1,\,sds=d\inv}$$
and we can describe complex representations of $D_{2\ell}$ by specifying pairs of matrices $(A,B)$ representing $d$ and $s$, respectively.
We define one-dimensional representations
$$\rho_0:\, (1,1),\quad \rho_1:\, (1,-1),\quad \rho_2:\, (-1,1),\quad \rho_3:\, (-1,-1)$$
and two-dimensional representations
$$\tau_k:\, \rbr{\pmat{z^k&0\\0&z^{-k}},\ \pmat{0&1\\1&0}},\qquad k\in\bZ,$$
where $z=e^{\pi i/\ell}$.
Note that $\ta_k\iso\tau_{2\ell-k}$.
The following lemma summarizes the well-known representation theory of $D_{2\ell}$:

\begin{lemma}
The irreducible representations of $D_{2\ell}$ are $\rho_k$ for $k=0,1,2,3$ and $\tau_k$ for $k=1,\ldots,\ell-1$.
We have
$$\tau_0\simeq\rho_0\oplus\rho_1,\qquad
\tau_\ell\simeq\rho_2\oplus\rho_3,$$
$$\rho_1\otimes\rho_1\simeq\rho_0\simeq\rho_2\otimes\rho_2,
\qquad\rho_1\otimes\rho_2\simeq\rho_3,$$
$$\tau_1\otimes\tau_k\simeq\tau_{k-1}\oplus\tau_{k+1},
\qquad k\in\bZ,$$
$$\rho_0\otimes\tau_k\simeq\tau_k\simeq\rho_1\otimes\tau_k,
\qquad k\in\bZ,$$
$$\rho_2\otimes\tau_k\simeq\tau_{\ell-k}\simeq\rho_3\otimes\tau_k,
\qquad k\in\bZ.$$
\end{lemma}

Consider the standard embedding
$D_{2\ell}\sbs\SO(3)\sbs\SL_3(\bC)$ and the induced representation~ $V$.
In this embedding $d$ acts as rotation in the $\ang{e_1,e_2}$-plane, and $s$ acts as inversion in the $\ang{e_2,e_3}$-plane:
$$d\mto\pmat{A&0\\0&1},\qquad s\mto\pmat{B&0\\0&-1},
\qquad 
A=\pmat{a&-b\\b&a},\qquad
B=\pmat{1&0\\0&-1},
$$
where $z=a+bi=e^{\pi i/{\ell}}$.
Note that in the third coordinate we get the representation $\rho_1$.
Taking the matrix $M=\smat{1&1\\-i&i}$ (its columns are the eigenvectors of $A$), we can see that
$$M\inv A M=\pmat{z&0\\0&z\inv},\qquad 
M\inv BM=\pmat{0&1\\1&0}.$$
This implies that the representation $(A,B)$ is isomorphic to $\ta_1$ and $V\iso \rho_1\oplus\tau_1$.

\begin{remark}
For $\ell=1$, we have $D_{2\ell}=\bZ_2\xx\bZ_2$ and the standard embedding $\bZ_2\xx\bZ_2\sbs\SO(3)$ is given by
$$d=(1,0)\mto\diag(-1,-1,1),\qquad s=(0,1)\mto\diag(1,-1,-1).$$
\end{remark}

Based on these facts, we can now describe the McKay quiver for $D_{2\ell}$ and its potential:

\begin{lemma}\label{lm:McKay}
The McKay quiver has the following form, where blue arrows correspond to $\rho_1$ and black arrows correspond to $\ta_1$.

For $\ell\ge 2$:
\tikzcdset{every arrow/.append style=<->}
\begin{ctikzcd}
\rho_0\ar[dr]\ar[dd,blue]&&&&&\rho_2\ar[dd,blue]\\
&\ta_1\rar\ar[->,loop above,blue]&
\ta_2\rar\ar[->,loop above,blue]&
\dots\rar&
\ta_{\ell-1}\ar[ru]\ar[rd]\ar[->,loop above,blue]\\
\rho_1\ar[ur]&&&&&\rho_3
\end{ctikzcd}

For $\ell=1$:

\begin{ctikzcd}
\rho_0\rar[<->]\dar[<->,blue]\drar[<->] &\rho_2\dar[<->,blue]\\
\rho_1\rar[<->]\urar[<->]&\rho_3
\end{ctikzcd}

\tikzcdset{every arrow/.append style=->}

The potential is a linear combination of cubic terms involving one blue arrow and two black arrows.
\end{lemma}
\begin{proof} Using the previous lemma, we find the decompositions
$$\rho_0\otimes V=\rho_0\otimes(\rho_1\oplus\tau_1)\simeq\rho_1\oplus\tau_1,$$
$$\rho_1\otimes V=\rho_1\otimes(\rho_1\oplus\tau_1)\simeq\rho_0\oplus\tau_1,$$
$$\rho_2\otimes V=\rho_2\otimes(\rho_1\oplus\tau_1)\simeq\rho_3\oplus\tau_{\ell-1},$$
$$\rho_3\otimes V=\rho_3\otimes(\rho_1\oplus\tau_1)\simeq\rho_2\oplus\tau_{\ell-1},$$
$$\tau_k\otimes V=\tau_k\otimes(\rho_1\oplus\tau_1)\simeq\tau_k\oplus(\tau_{k-1}\oplus\tau_{k+1})$$
for $k=2,\ldots,\ell-2$, and
$$\tau_1\otimes V=\tau_1\otimes(\rho_1\otimes\tau_1)\simeq\tau_1\oplus(\rho_0\oplus\rho_1\oplus\tau_2),$$
$$\tau_{\ell-1}\otimes V=\tau_{\ell-1}\otimes(\rho_1\oplus\tau_1)\simeq\tau_{\ell-1}\oplus(\tau_{\ell-2}\oplus\rho_2\oplus\rho_3).$$
This yields the above McKay quivers.
To describe the potential, we decompose $V^{\otimes 3}$ into irreducible representations and single out the contributions of the trivial representation $\rho_0$. Namely, we find
$$V^{\otimes 3}\simeq\rho_1^{\otimes 3}\oplus(\rho_1^{\otimes 2}\otimes\tau_1)^{\oplus 3}\oplus(\rho_1\otimes\tau_1^{\otimes 2})^{\oplus 3}\oplus\tau_1^{\otimes 3},$$
and
$$\rho_1^{\otimes 3}\simeq\rho_1,\qquad
\rho_1^{\ts2}\otimes\tau_1\simeq\tau_1,$$
$$\rho_1\otimes\tau_1^{\otimes2}
\iso\rho_1\ts(\rho_0\oplus\rho_1\oplus\ta_2)
\simeq\rho_1\oplus\rho_0\oplus\tau_2,$$
$$\tau_1^{\otimes 3}
\iso\tau_1\ts(\rho_0\oplus\rho_1\oplus\ta_2)
\simeq\tau_1^{\oplus 3}\oplus\tau_3.$$
We analyse the determinant map $\det:V^{\otimes 3}\to\bigwedge^3V$ with respect to this direct sum decomposition. Since, obviously, $\bigwedge^3V\iso\rho_0$, and $\rho_1\otimes\tau_1^{\otimes 2}$ is the only summand containing a copy of $\rho_0$, the determinant map
factors through the projection $V^{\otimes 3}\rightarrow(\rho_1\otimes\tau_1^{\otimes 2})^{\oplus 3}$. Note that this holds true independent of $\ell$. By definition of the potential, this implies that the potential involves only triangles of arrows involving one blue arrow (corresponding to~ $\rho_1$) and two black arrows (corresponding to $\ta_1$). 
\end{proof}

\begin{remark} For odd dihedral groups $D_{2\ell+1}$ similar computations can be performed, leading to the McKay quiver
\tikzcdset{every arrow/.append style=<->}
\begin{ctikzcd}
\rho_0\ar[dr]\ar[dd,blue]\\
&\ta_1\rar\ar[->,loop above,blue]&
\ta_2\rar\ar[->,loop above,blue]&
\dots\rar&
\ta_{\ell}\ar[->,loop above,blue]
\ar[loop right,->]\\
\rho_1\ar[ur]
\end{ctikzcd}
\tikzcdset{every arrow/.append style=->}
and potential as before. It admits a cut consisting of blue arrows, but the resulting quiver with relations involves a square of the loop at $\ta_\ell$. Therefore, it does not admit a second cut, and our method for factoring the motivic generating function does not apply.
\end{remark}

\sec[Generating series via double dimensional reduction] Based on the dimensional reduction method described in \S\ref{potentialwithcut} we can now describe the motivic generating series in terms of the representation theory of quivers of extended Dynkin types.
Let $(Q,W)$ be the McKay quiver with potential described in Lemma \ref{lm:McKay}. 
There is a cut $I\sbs Q_1$ consisting of all black arrows going from right to left and a cut $I'\sbs Q_1$ consisting of all blue arrows.
Let $Q'=Q\ms I$ and 
$J_I=\bC Q'/(\dd W/\dd a\rcol a\in I)$.
Then the motivic generating series \eqref{total2} has the form
\begin{gather}
\cA(t)=\sum_d (-q^\oh)^{s(d,d)}\frac{[R(J_I,d)]}{[G_d]}t^d,\\
s(d,d)=\hi_Q(d,d)+2\sum_{(a:i\to j)\in I}d_id_j
=\hi_{C_2\sqcup C_2}(d,d),
\end{gather}
where $C_2\sqcup C_2$ is the restriction of $Q$ to the vertices $\rho_i$ (and blue arrows between them).
It is indeed just a union of two cyclic quivers.

Consider the quiver $Q''=Q\ms(I\cup I')$, which is an extended Dynkin quiver of type $\hat D_{\ell+2}$ for $\ell\geq 2$ (respectively of type $\hat D_3=\hat A_3$ for $\ell=1$).
The projection $\pi:R(J_I,d)\to R(Q'',d)$ has the fiber over 
$M\in R(Q'',d)$ encoded by the blue arrows.
More precisely, this fiber is equal to $\Hom(M,\Si M)$, where 
\begin{equation}\label{Sigma}
\Si:Q''\to Q''
\end{equation}
is the involution on $Q''$ induced by the blue arrows.
Applying \eqref{part function},
we can rewrite the motivic generating series as a sum over representations of $Q''$:
\begin{lemma}\label{afterdoublecut}
We have\begin{multline*}
\cA(t)
=\sum_{d\in\bN Q_0}\sum_{\udim M=d}
\frac{(-q^\oh)^{s(d,d)+2h^0(M,\Si M)}}{[\Aut M]}t^{d}\\
=\sum_{m:S\to\bN}
\frac{(-q^\oh)^{-\sum_{X,Y}m_Xm_Y\si(X,Y)}}
{\prod_X(q\inv)_{m_X}}t^{\sum_X m_X\udim X},
\end{multline*}
where 
$$\si(M,N)=2h^0(M,N)-2h^0(M,\Si N)-s(M,N),$$
and $S$ denotes the set of isomorphism classes of indecomposable $Q''$-representations.
\end{lemma}

\begin{remark}
The above quiver with potential $(Q,W)$ is an example of a translation potential quiver studied in \cite{mozgovoy_translation}.
Given a quiver $Q$ with an automorphism $\vi:Q\to Q$,
we define a twisted double quiver $Q^\vi$ by adding arrows $a^*:\vi j\to i$, for all arrows $a:i\to j$ in~ $Q$.
We define a quiver $\tl Q^\vi$ by adding, furthermore, arrows $\ell_i:i\to\vi i$, for all vertices $i$ in~ $Q$.
The quiver $\tl Q^\vi$ is equipped with a potential consisting of cycles of the form
$$i\xto{\ell_i}\vi i\xto{\vi a}\vi j\xto{a^*} i,
\qquad i\xto{a} j\xto{\ell_j}\vi j\xto{a^*} i$$
for arrows $a:i\to j$ in $Q$.
Applying this construction to the quiver $Q''$ with the automorphism~$\Si$, we obtain the McKay quiver $Q$
with the potential from Lemma \ref{lm:McKay}.
\end{remark}

%% file: computation_of_sigma.tex
\section{Computation of \tpdf{$\sigma$}{sigma}}
\label{sec:sigma}
To describe the motivic generating series $\cA(t)$
using its expression from Lemma~\ref{afterdoublecut}, it is necessary to compute the interaction form $\sigma$ (or its symmetrization) on indecomposable representations of the quiver $Q''$ of type $\hat D_{\ell+2}$.
To simplify notation, we will use $r=\ell+2\ge 3$
and we will index the vertices of the quiver $Q''$ of type $\hat D_{r}$ as
\begin{ctikzcd}
0\drar&&&&&&r-1\\
&2\rar&3\rar&\dots\rar&r-3\rar&r-2\urar\drar\\
1\urar&&&&&&r
\end{ctikzcd}
for $r\ge4$ and
\begin{ctikzcd}
0\rar\drar&2\\
1\rar\urar&3
\end{ctikzcd}
for $r=3$ (this is a quiver of type $\hat D_3=\hat A_3$).
We first recall standard facts on the representation theory of extended Dynkin quivers of type $\hat D_{r}$ from \cite[Section 6]{dlabringel}.

\subsection{Root system \tpdf{$\hat D_r$}{Dr}}
The root system of finite type $D_r$ is constructed as follows.
Consider a vector space with a basis $(\eps_1,\dots,\eps_r)$ and a scalar product $(\eps_i,\eps_j)=\de_{ij}$.
Then the set of (positive) roots is 
$\De^\fin_+=\sets{\eps_i\pm\eps_j}{i<j}$
and
the simple roots are $\al_i=\eps_i-\eps_{i+1}$ for $1\le i<r$ and $\al_r=\eps_{r-1}+\eps_r$.
The maximal root is $\eps_1+\eps_2$.
To construct the root system of type $\hat D_r$, we consider an additional simple root $\al_0$ so that the indivisible imaginary root has the form
\begin{equation}
\de=\al_0+\al_1+2(\al_2+\dots+\al_{r-2})+\al_{r-1}+\al_r.
\end{equation}
The set of positive real roots of type $\hat D_r$ is
\begin{equation}
\De^\re_+=\De^\fin_+\cup\sets{\De^\fin+n\de}{n\ge1},
\qquad \De^\fin=\De^\fin_+\cup(-\De^\fin_+).
\end{equation}
The set of positive imaginary roots is $\De^\im_+=\sets{n\de}{n\ge1}$.

In what follows we will identify standard basis vectors $e_i\in\bZ^{Q''_0}$ with the simple roots $\al_i$
and we will identify a vector $d\in\bZ^{Q''_0}$ with the vector $\sum_{i=0}^r d_i\al_i$ in the root lattice.
Note that dimension vectors of indecomposable representations of $Q''$ can be identified with the positive roots of the root system of type $\hat D_{r}$.
We define
\begin{equation}
\rho=e_0+e_1+\dots+e_r.
\end{equation}

\subsection{AR translation}
Recall that, for any acyclic quiver $Q$, 
the AR translation 
\begin{equation}
\ta:D^b(\Rep Q)\to D^b(\Rep Q)
\end{equation}
is given by $\ta=\nu[-1]$, 
where $\nu(M)=\RHom(M,A)\dual$, $A=\bC Q$.
We will also denote by $\ta$ the induced map on the Grothendieck group $K(Q)\iso\bZ^{Q_0}$.
For any vertex $i\in Q_0$, let $P_i=Ae_i$ 
be the corresponding indecomposable projective (left) $A$-module and $I_i=(e_iA)\dual$ be the corresponding indecomposable injective (left) $A$-module.
Then $\nu(P_i)=I_i$.

Recall that the involution $\Si$ of the quiver $Q''$
defined in \eqref{Sigma}
acts on the vertices by 
\begin{equation}
\Sigma(0)=1,\qquad
\Sigma(r-1)=r,\qquad
\Sigma(i)=i,\quad 2\leq i\leq r-2.
\end{equation}
It induces a functor
\begin{equation}
\Si:\Rep Q''\to \Rep Q''
\end{equation}
as well as an automorphism of the Grothendieck group
$K(Q'')$.
We have $\nu\Si P_i=I_i=\Si\nu P_i$ for $2\le i\le r-2$,
$\nu\Si P_0=I_1=\Si \nu P_0$ and similarly for $P_1,P_{r-1}$, $P_{r}$.
This implies that
$$\ta\Si=\Si\ta.$$

\begin{remark}
The fact that $\nu=\ta[1]$ is the Serre functor implies that $\ta$ commutes with any exact auto-equivalence of $D^b(\Rep Q)$.
\end{remark}

\begin{lemma}
The map $\ta$ on $\bZ^{Q''_0}$ is given by (with $\rho=\sum_{i=0}^r e_i$)
$$e_2\mto e_3\mto\dots\mto e_{r-2}\mto\rho\mto e_2,$$
$$e_{r-1}\mto-\rho+e_r,\quad
e_{r}\mto-\rho+e_{r-1},\quad
e_0\mto e_1+e_2,\quad
e_1\mto e_0+e_2.
$$
\end{lemma}
\begin{proof}
For $2\le k\le r-2$, we have $\ta(P_k)=\ta(\sum_{i\ge k}e_i)=-\sum_{i\le k}e_i$.
Therefore $\ta(e_k)=e_{k+1}$ for $2\le k<r-2$.
On the other hand $\ta(e_{r-2})=-(I_{r-2}-I_{r-1}-I_{r})=\rho$.

We have $\ta(\rho-e_1)=\ta(P_0)=-I_0=-e_0$.
Similarly $\ta(\rho-e_0)=-e_1$.
On the other hand $\ta(\rho-e_0-e_1)=\ta(P_2)=-(e_0+e_1+e_2)$.
Therefore $\ta(\rho)=e_2$, $\ta(e_0)=e_1+e_2$, $\ta(e_1)=e_0+e_2$.
Finally, $\ta(e_{r-1})=-I_{r-1}=-\rho+e_r$.
Similarly $\ta(e_r)=-\rho+e_{r-1}$.
\end{proof}

\begin{remark}\label{alternation}
Note that
$$\ta(e_0-e_1)=e_1-e_0,\qquad \ta(e_{r-1}-e_r)=e_r-e_{r-1}.$$
Note also that $\de=\rho+e_2+\dots+e_{r-2}$ and $\ta(\de)=\de$.
\end{remark}

\subsection{Properties of indecomposable representations}
The preprojective indecomposables of $Q''$ are of the form $\tau^{-k}P_i$, where
$k\geq 0$, $i$ is a vertex of $Q''$ and
$\ta$ is the AR translation.
Similarly, the preinjective indecomposables are of the form $\tau^kI_i$. 
The remaining indecomposables belong to tubes, namely,
an infinite family of homogeneous tubes $\mathcal{R}(\lambda)$ (with the trivial action of $\ta$),
two tubes $\mathcal{R}_1,\mathcal{R}_2$ of width $2$ 
(with the action of $\ta$ having order~ $2$), and a tube $\mathcal{R}_3$ of width $r-2$ 
(respectively, no such tube in the case $r=3$).
We have the following dimension vectors of regular indecomposable representations:
\begin{description}
\item[$\mathcal{R}(\lambda)$] $\udim X=\delta$,
\item[$\mathcal{R}_1$] $\udim X_1=e_0+e_2+\ldots+e_{r-2}+e_{r-1}$, $\udim X_2=e_1+e_2+\ldots+e_{r-2}+e_{r}$,
\item[$\mathcal{R}_2$] $\udim X_1=e_0+e_2+\ldots+e_{r-2}+e_{r}$, $\udim X_2=e_1+e_2+\ldots+e_{r-2}+e_{r-1}$,
\item[$\mathcal{R}_3$] 
$\udim X_1=e_2$, 
$\udim X_2=e_3$, $\ldots$, $\udim X_{r-3}=e_{r-2}$, $\udim X_{r-2}=\rho=\sum_{i=0}^{r} e_i$.
\end{description}

%
%

Inspection of the above dimension vectors of the regular indecomposables shows that $\Sigma$ acts as identity on $\mathcal{R}(\lambda)$ and $\mathcal{R}_3$, and acts as $\tau$ on $\mathcal{R}_1,\mathcal{R}_2$.


As before, we define $h^i(M,N)=\dim\Ext^i(M,N)$.
Recall from Lemma \ref{afterdoublecut} that
\begin{equation}
\si(M,N)=2h^0(M,N)-2h^0(M,\Si N)-s(M,N).
\end{equation}

\begin{lemma}\label{si as h01}
For all representations $M$, $N$ of $Q''$
we have
$$\si(M,N)=h^0(M,N)+h^1(M,N)-h^0(M,\Si N)-h^1(M,\Si N).$$
In particular, $\si(\ta M,\ta N)=\si(M,N)$.
\end{lemma}
\begin{proof}
Using the definitions of the Euler forms $\hi$ of $Q''$ and $s$ of $C_2\sqcup C_2$, we can directly verify that
$$\hi(d,e)-\hi(d,\Sigma e)=s(d,e).$$
This implies the statement.
\end{proof}

\begin{lemma}\label{lm:prop of si}
For all representations $M$, $N$ of $Q''$
we have
\begin{enumerate}
\item $\sigma(M,\Sigma N)
=\si(\Si M,N)=-\sigma(M,N)$.
\item
$\sigma(M,N)=\sigma(N,\tau M)$.
\item $\si(M,N)+\si(N,M)=\si(M,N\oplus \ta N)$.
\end{enumerate}
\end{lemma}
\begin{proof}
The first statement follows immediately from Lemma \ref{si as h01}.
Using the Auslander-Reiten formula $h^1(M,N)=h^0(N,\tau M)$, we obtain
$$\si(M,N)=h^0(M,N)+h^0(N,\ta M)-h^0(M,\Si N)-h^0(N,\Si\ta M)$$
Therefore
$$\si(N,\ta M)
=h^0(N,\ta M)+h^0(\ta M,\ta N)
-h^0(N,\Si\ta M)-h^0(\ta M,\Si\ta N)$$
and the second statement follows.
The third statement is a direct consequence of the second statement.
\end{proof}

\begin{corollary}\label{cor:vanishing si}
If $\Si M\iso M$ or $\Si M\iso\ta M$, then $\si(M,N)+\si(N,M)=0$. 
\end{corollary}
\begin{proof}
If $\Si M\iso M$, then $\si(M,N)=\si(\Si M,N)=-\si(M,N)$, hence $\si(M,N)=0$. 
Similarly $\si(N,M)=0$.
If $\si M\iso\ta M$, then $\si(M,N)=\si(N,\Si M)=-\si(N,M)$,
hence $\si(M,N)+\si(N,M)=0$.
\end{proof}



These facts allow us to compute the symmetrization of $\sigma$ on arbitrary pairs of indecomposables:

\begin{lemma}\label{formulasigma}
If $M,N$ are indecomposable representations of $Q''$, then
\sloppy
$\sigma(M,N)+\sigma(N,M)\not=0$ only when $M$ is in the $\tau$-orbit of $P_i$ or $I_i$ for $i=0,1,r-1,r$, and $N=M$ or $N=\Sigma M$. In this case, we have 
$$\sigma(M,M)=1,\qquad
\sigma(M,\Sigma M)=\sigma(\Sigma M,M)=-1.$$
\end{lemma}

\begin{proof}
According to our analysis of indecomposables, $\Si$ acts as identity or as $\ta$ on all regular indecomposable representations.
It also acts as identity on representations of the form $\ta^{-k}P_i,\ta^k I_i$ for all $k\ge0$, $2\le i\le r-2$.
Therefore by the above corollary, we can assume that $M,N$ are in the $\ta$-orbits of $P_i$ or $I_i$ for $i=0,1,r-1,r$.
We only treat the case where $M$ is preprojective, the other case then follows by duality.
By $\tau$-invariance of $\si$, we can assume that $M=P_i$.
Then we obtain from Lemma \ref{si as h01}
\begin{equation}\label{proj si}
\sigma(P_i,N)=h^0(P_i,N)-h^0(P_i,\Si N)=
d_i-d_{\Sigma i},
\end{equation}
where $d$ is the dimension vector of $N$.
This functional on dimension vectors alternates under $\tau$
by Remark \ref{alternation}, thus assumes the value $0$ on a dimension vector of the form $d+\tau d$.
If $N$ is not projective, then \eqref{proj si} applies to $N\oplus \ta N$, hence $\si(P_i,N\oplus \ta N)=0$
and we conclude that $\si(P_i,N)+\si(N,P_i)=0$
by Lemma \ref{lm:prop of si}.
Let $N=P_j$ be projective.
If $2\le j\le r-2$, then $\Si P_j=P_j$ and we conclude that
$\si(P_i,P_j)+\si(P_j,P_i)=0$ by Corollary 
\ref{cor:vanishing si}.
Therefore we can assume that $j=0,1,r-1,r$.

If $j\ne i,\Si i$, then 
$\si(P_i,P_j)=0$ by \eqref{proj si} and similarly $\si(P_j,P_i)=0$.

If $j=i$, then $\si(P_i,P_i)=1$ by \eqref{proj si}.

If $j=\Si i$, then $\si(P_i,P_{\Si i})=-1=\si(P_{\Si i}, P_i)$ by \eqref{proj si}.
\end{proof}


\subsection{Characterization of some \tpdf{\ta}{tau}-orbits}
In this section we will study dimension vectors of indecomposable representations in the above $\ta$-orbits.

\begin{lemma}
We have (in the Grothendieck group of $Q''$)
$$\Si\ta(I_k)=e_2+I_k,\qquad k=0,1,r-1,r.$$
$$\ta(I_k)=e_2+I_{k+1},\qquad 2\le k<r-2.$$
$$\ta(I_{r-2})=e_2+\rho+I_{r-2}.$$
\end{lemma}
\begin{proof}
We have $\ta(I_0)=\ta(e_0)=e_1+e_2$, hence $\Si\ta(I_0)=e_0+e_2=e_2+I_0$.
We have
$$\ta(I_{r})=\ta(\rho-e_{r-1})
=e_2+\rho-e_{r}=
e_2+I_{r-1},
$$
hence $\Si\ta I_r=e_2+I_r$.
For $2\le k<r-2$, we have
$$\ta(I_k)=\ta\bigg(\sum_{i\le k}e_i\bigg)
=e_2+\sum_{i\le k+1}e_i=e_2+I_{k+1}.$$
Finally,
$\ta(I_{r-2})=\ta(\rho-e_{r-1}-e_r)
=e_2+2\rho-e_r-e_{r-1}=e_2+\rho+I_{r-2}$.
\end{proof}


\begin{lemma}
\label{lm:p=0}
An indecomposable representation having dimension vector $d$ is in a
\ta-orbit of $P_i,I_i$ for $i=0,1,r-1,r$ if and only if $p(d)=1$, where
$$p(d)=d_0+d_1+d_{r-1}+d_r\pmod 2.$$
\end{lemma}
\begin{proof}
We conclude from the description of regular indecomposable representations that their dimension vectors satisfy $p(d)=0$.
For $2\le k\le r-2$ we have
\begin{multline*}
I_k\xmto{\ta} e_2+I_{k+1}
\xmto{\ta} e_2+e_3+I_{k+2}
\xmto{\ta}\dots\xmto{\ta}
e_2+\dots+e_{r-k-1}+I_{r-2}\\
\xmto{\ta} e_2+\dots+e_{r-k}+\rho+I_{r-2}
=\de+I_{r-k}.
\end{multline*}
Note that all of these classes have dimension vectors satisfying $p(d)=0$.
The same applies to higher powers of $\ta$.
For $k=0,1,r-1,r$ and $\bta=\Si\ta$ we have
$$I_{k}\xmto{\bta}e_2+I_k
\xmto{\bta} e_2+e_3+I_k\xmto{\bar\ta}\dots\xmto{\bar\ta}
e_2+\dots+e_{r-2}+I_k\xmto{\bar\ta}\de+I_{k}.$$
Note that all of these classes have dimension vectors satisfying $p(d)=1$.
The same applies to higher powers of $\bar\ta$.
This implies that also the dimension vectors of $\ta^n I_k$ for $n\ge0$, $k=0,1,r-1,r$ satisfy $p(d)=1$.
The proof for pre-projective representations is the same.
\end{proof}


%% file: generating_function.tex
\section{Motivic generating series}\label{sec:genfun}
Recall from Lemma \ref{afterdoublecut} that
\begin{equation}
\cA(t)=\sum_{m:S\to\bN}
\frac{(-q^\oh)^{-\sum_{X,Y}m_Xm_Y\si(X,Y)}}
{\prod_X(q\inv)_{m_X}}t^{\sum_X m_X\udim X}
\end{equation}
where $S$ is the set of isomorphism classes of indecomposable representations of the quiver~$Q''$
of type $\hat D_{\ell+2}$
and $\si$ is the interaction form computed in \S\ref{sec:sigma}.
As in Lemma \ref{lm:p=0}, we define
\begin{equation}
p(d)=d_0+d_1+d_{\ell+1}+d_{\ell+2}\pmod 2.
\end{equation}
We will present an explicit formula for the motivic generating series $\cA(t)$ in terms of the root system of type $\hat D_{\ell+2}$.

\begin{theorem}
\label{thm:gen fun}
We have
$$\cA(t)=\Exp\rbr{\frac{\sum_d\Om_d(q)t^d}{q-1}},$$
where $\Om_d(q)$ is
\begin{enumerate}
\item $q$ if $d\in \De^\re_+$ with $p(d)=0$ or $d=d'+\Si d'$ for some $d'\in\De^\re_+$ with $p(d')=1$.
\item $-q^\oh$ if $d\in\De_+^\re$ with $p(d)=1$.
\item $q(q+\ell+2)$ if $d\in\De_+^\im$.
\item zero otherwise. 
\end{enumerate}
\end{theorem}

\begin{remark}
Let $\hi$ be the Euler form of $Q''$.
As $Q''$ is of affine type, a vector $d\in\bZ^{Q''_0}$ is a real root if and only if $\hi(d,d)=1$ 
and it is an imaginary root if and only if $\hi(d,d)=0$
(see \eg \cite[Prop.~1.6]{kac_infinite}).
In particular, if $d=d'+\Si d'$, for some $d'\in\De_+^\re$,
then $\hi(d,d)=2\hi(d',d')+2\hi(d',\Si d')$ is even,
hence $d$ is not a real root.
If we also have $p(d')=1$, then $d'$ is the dimension vector of $\ta^{-k}P_i$ or $\ta^k I_i$ for some $i=0,1,\ell+1,\ell+2$ and $k\ge0$.
Assuming that $d'=\udim \ta^{-k}P_i$, we have
$\hi(d',d')=\hi(P_i,P_i)=1$ and similarly $\hi(d',\Si d')=0$.
Therefore $\hi(d,d)=2$ and $d$ is not a root.
\end{remark}

\begin{remark}[The case of $\bZ_2\xx\bZ_2$]
For the root system of type $\hat A_3$, the finite positive roots are $$\De^\fin_+=\set{\al_1,\al_2,\al_3,\al_1+\al_2,\al_1+\al_3,\al_1+\al_2+\al_3}$$
and the indivisible imaginary root is $\de=\al_0+\al_1+\al_2+\al_3$.
The roots such that $p(d)=1$ are $$\set{\al_0,\al_2,\al_0+\al_2+\al_3,\al_0+\al_1+\al_2}
+\bN\de$$
and their $\Si$ translations.
This implies that $\Om_d(q)=q$ for $d\in\set{\al_0+\al_1,\al_2+\al_3}+\bN\de$ (note that these elements are not roots).
We also have $\Om_d(q)=q$ for the roots 
$$d\in\set{\al_1+\al_2,\al_1+\al_3,\al_0+\al_3,\al_0+\al_2}+\bN\de$$
satisfying $p(d)=0$.
We conclude that $\Om_d(q)=q$, for all $d=\al_i+\al_j+k\de$ with $i\ne j$, $k\ge0$.
Therefore (identifying $Q''_0$ with $\bZ_4$) 
\begin{equation}
\cA(t)
=\Exp\rbr{\frac{
q\sum_{i< j}t_it_{j}
-q^\oh\sum_i(t_i+t_it_{i+1}t_{i+2})
+q(q+3)t^\de
}{q-1}\sum_{n\ge0}t^{n\de}
}.
\end{equation}
We can use the above formula to compute unrefined
non-commutative Donaldson-Thomas (NCDT) invariants
and compare the result to \cite{young_generating}, where the molten crystal interpretation of NCDT invariants was used.
More precisely, the generating function of NCDT invariants (framed at the vertex~$0$) can be written as (\cf \cite{MP})
\begin{equation}
Z_{0,\NCDT}(-t_0,t_1,t_2,t_3)
=\Exp\rbr{\sum_d d_0\Om_d(1)t^d}.
\end{equation}
Applying our formulas for the motivic DT invariants $\Om_d(q)$, we can write the above expression as an infinite product 
\begin{gather}
Z_{0,\NCDT}(-t_0,t_1,t_2,t_3)
=M(1,t^\de)^4\cdot
\frac{\tl M(t_1t_2,t^\de)\tl M(t_1t_3,t^\de)
\tl M(t_2t_3,t^\de)}
{\tl M(t_1t_2t_3,t^\de)\tl M(t_1,t^\de)\tl M(t_2,t^\de)\tl M(t_3,t^\de)},
\\
M(q,t)=\prod_{n\ge1}(1-qt^n)^{-n},\qquad 
\tl M(q,t)=M(q,t)M(q\inv,t),
\end{gather}
which coincides with \cite[Theorem 1.7]{young_generating} up to some typos in loc.~cit.
\end{remark}

\begin{proof}[Proof of Theorem \ref{thm:gen fun}]
As in the previous section, we will use $r=\ell+2$.
We have seen earlier that (the symmetrization of) $\si(X,Y)$ can be nonzero only if $X,Y$ are in the \ta-orbits of $P_i,I_i$ for $i=0,1,r-1,r$.
The contribution of every indecomposable $X$ not of this form~is
$$\sum_{m\ge0}\frac1{(q\inv)_m}t^{m\udim X}
=\Exp\rbr{\frac{qt^{\udim X}}{q-1}}.$$

Note that the family of regular representations having dimension vector $d=k\de$ (imaginary root) has the motivic class $q+r$. Therefore their contribution is (we take the plethystic power)
$$\Exp\rbr{\frac{qt^d}{q-1}}^{q+r}
=\Exp\rbr{\frac{q(q+r)t^d}{q-1}}.
$$

According to \S\ref{sec:sigma},
an indecomposable representation $M$ having dimension vector $d$ with $p(d)=1$ has a nonzero interaction only with itself and a representation $\Si M$, having dimension vector ~$\Si d$.
Moreover, $\si(M,M)=1=-\si(M,\Si M)$.
By Theorem \ref{C2 part function},
the corresponding contribution to the generating series is equal to
$$F(t^d,t^{\Si d})=\Exp\rbr{\frac{qt^{d+\Si d}-q^\oh(t^d+t^{\Si d})}{q-1}}.$$
Summarizing, we obtain
\begin{multline*}
\cA(t)
=
\prod_{\ov{d\in\De^\re_+}{p(d)=0}}
\Exp\rbr{\frac{qt^d}{q-1}}\cdot
\prod_{d\in\De^\im_+}
\Exp\rbr{\frac{q(q+r)t^d}{q-1}}\xx\\
\xx\prod_{\ov{d\in\De^\re_+/\Si}{p(d)=1}}
\Exp\rbr{\frac{qt^{d+\Si d}-q^\oh(t^d+t^{\Si d})}{q-1}}.
\end{multline*}
This finishes the proof of Theorem \ref{thm:gen fun}.
\end{proof}